\documentclass[11pt]{amsart}
\usepackage{amssymb}
\usepackage{amsthm,amsmath}
\usepackage{epic,cite}

\title[]{Characterizations of quasitrivial symmetric nondecreasing associative operations}

\author{Jimmy Devillet}
\address{Mathematics Research Unit, University of Luxembourg, Maison du Nombre, 6, avenue de la Fonte, L-4364 Esch-sur-Alzette, Luxembourg}
\email{jimmy.devillet[at]uni.lu}

\author{Gergely Kiss}
\address{Alfred Renyi Institute of Mathematics, Hungarian Academy of Science, HU-1053 Budapest, Re\'altanoda u.\ 13-15, Hungary}
\email{kigergo57[at]gmail.com}

\author{Jean-Luc Marichal}\thanks{Communicated by Mikhail Volkov}
\address{Mathematics Research Unit, University of Luxembourg, Maison du Nombre, 6, avenue de la Fonte, L-4364 Esch-sur-Alzette, Luxembourg}
\email{jean-luc.marichal[at]uni.lu}

\date{}

\theoremstyle{plain}
\newtheorem{theorem}{Theorem}[section]
\newtheorem{lemma}[theorem]{Lemma}
\newtheorem{proposition}[theorem]{Proposition}
\newtheorem{corollary}[theorem]{Corollary}

\theoremstyle{definition}
\newtheorem{definition}[theorem]{Definition}

\theoremstyle{remark}
\newtheorem{remark}{Remark}
\newtheorem{example}{Example}

\newcommand{\bfx}{\mathbf{x}}
\newcommand{\bfy}{\mathbf{y}}
\newcommand{\R}{\mathbb{R}}
\newcommand{\Cdot}{\boldsymbol{\cdot}}

\begin{document}
\begin{abstract}
We provide a description of the class of $n$-ary operations on an arbitrary chain that are quasitrivial, symmetric, nondecreasing, and associative. We also prove that associativity can be replaced with bisymmetry in the definition of this class. Finally we investigate the special situation where the chain is finite.
\end{abstract}



\maketitle

%---------------------------------------------------------------------------------------
\section{Introduction}

Let $X$ be a nonempty set and let $n\geq 2$ be an integer. The $n$-ary operations $F\colon X^n\to X$ satisfying the associativity property (see Definition~\ref{de:def1} below) have been extensively investigated since the pioneering work by D\"ornte \cite{Dor28} and Post \cite{Pos40}. In the algebraic language, when $F$ is such an operation, the pair $(X,F)$ is called an $n$-ary semigroup. For a background on this topic see, e.g., \cite{DudMuk96,DudMuk06,Vin61} and the references therein.

In this paper we investigate the class of $n$-ary operations $F\colon X^n\to X$ on a chain $(X,\leq)$ that are quasitrivial, symmetric, nondecreasing, and associative (quasitriviality means that $F$ always outputs one of its input values). After presenting some definitions and preliminary results in Section 2, we provide in Section 3 a characterization of these operations and show that they are derived from associative binary operations (Theorems~\ref{thm:main2}, \ref{thm:f456dfs}, and Corollary~\ref{cor:fs5fs}). We also discuss the special situation where these operations have neutral elements (Proposition~\ref{prop:eee}), in which case they are derived from the so-called idempotent uninorms (Corollary~\ref{cor:fs6}). In Section 4 we investigate certain bisymmetric $n$-ary operations and derive a few equivalences among properties involving associativity and bisymmetry. For instance we show that if an $n$-ary operation is quasitrivial and symmetric, then it is associative if and only if it is bisymmetric (Corollary~\ref{cor:24f}). This observation enables us to replace associativity with bisymmetry in some of our characterization results. Finally, in Section 5 we particularize our results to the special case where the operations are defined on finite chains (Corollary~\ref{cor:f456dfs2} and Theorem~\ref{thm:GC}).

We use the following notation throughout. A chain $(X,\leq)$ will simply be denoted by $X$ if no confusion may arise. For any chain $X$ and any $x,y\in X$, we use the symbols $x\wedge y$ and $x\vee y$ to represent $\min\{x,y\}$ and $\max\{x,y\}$, respectively. For any integer $k\geq 0$, we set $[k]=\{1,\ldots,k\}$. Finally, for any integer $k\geq 0$ and any $x\in X$, we set $k\Cdot x=x,\ldots,x$ ($k$ times). For instance, we have $F(3\Cdot x,2\Cdot y,0\Cdot z)=F(x,x,x,y,y)$.

%---------------------------------------------------------------------------------------
\section{Preliminaries}

In this section we introduce some basic definitions and present some preliminary results. Let $X$ be an arbitrary nonempty set.

\begin{definition}\label{de:def1}
An operation $F\colon X^n\to X$ is said to be
\begin{itemize}
\item \emph{idempotent} if $F(n\Cdot x)=x$ for all $x\in X$;
\item \emph{quasitrivial} (or \emph{conservative}, \emph{selective}) if $F(x_1,\ldots,x_n)\in\{x_1,\ldots,x_n\}$ for all $x_1,\ldots,x_n\in X$;
\item \emph{symmetric} if $F(x_1,\ldots,x_n)$ is invariant under any permutation of $x_1,\ldots,x_n$;
\item \emph{associative} if
\begin{eqnarray*}
\lefteqn{F(x_1,\ldots,x_{i-1},F(x_i,\ldots,x_{i+n-1}),x_{i+n},\ldots,x_{2n-1})}\\
&=& F(x_1,\ldots,x_i,F(x_{i+1},\ldots,x_{i+n}),x_{i+n+1},\ldots,x_{2n-1})
\end{eqnarray*}
for all $x_1,\ldots,x_{2n-1}\in X$ and all $i\in [n-1]$;
\item\emph{bisymmetric} if
$$
F(F(\mathbf{r}_1),\ldots,F(\mathbf{r}_n)) ~=~ F(F(\mathbf{c}_1),\ldots,F(\mathbf{c}_n))
$$
for all $n\times n$ matrices $[\mathbf{c}_1 ~\cdots ~\mathbf{c}_n]=[\mathbf{r}_1 ~\cdots ~\mathbf{r}_n]^T\in X^{n\times n}$.
\end{itemize}
If $(X,\leq)$ is a chain, then $F\colon X^n\to X$ is said to be
\begin{itemize}
\item \emph{nondecreasing} (for $\leq$) if $F(x_1,\ldots,x_n)\leq F(x'_1,\ldots,x'_n)$ whenever $x_i\leq x'_i$ for all $i\in [n]$.
\end{itemize}
\end{definition}

\begin{definition}
Let $F\colon X^n\to X$ be an operation.
\begin{itemize}
\item An element $e\in X$ is said to be a \emph{neutral element} of $F$ if
$$
F((i-1)\Cdot e,x,(n-i)\Cdot e) ~=~ x
$$
for all $x\in X$ and all $i\in [n]$. A neutral element need not be unique when $n\geq 3$ (e.g., $F(x_1,x_2,x_3)\equiv x_1+x_2+x_3~(\mathrm{mod}~2)$ on $X=\mathbb{Z}_2$).
\item The points $\bfx$ and $\bfy$ of $X^n$ are said to be \emph{connected for $F$} if $F(\bfx)=F(\bfy)$. The point $\bfx$ of $X^n$ is said to be \emph{isolated for $F$} if it is not connected to another point in $X^n$.
\end{itemize}
\end{definition}

\begin{lemma}\label{lemma:IdIs}
Let $F\colon X^n\to X$ be an idempotent operation. If $\bfx=(x_1,\ldots,x_n)\in X^n$ is isolated for $F$, then necessarily $x_1=\cdots =x_n$.
\end{lemma}

\begin{proof}
Let $\bfx=(x_1,\ldots,x_n)$ be isolated for $F$. From the identity $F(\bfx)=F(n\Cdot F(\bfx))$ we immediately derive $\bfx=(n\Cdot F(\bfx))$, that is, $x_1=\cdots =x_n=F(\bfx)$.
\end{proof}

\begin{remark}
We observe that any quasitrivial operation $F\colon X^n\to X$ has at most one isolated point. Indeed, such an operation $F$ is necessarily idempotent and hence the result follows from Lemma~\ref{lemma:IdIs} and the fact that we have
$$
F((n-1)\Cdot x,y) ~\in ~\{x,y\} ~=~ \{F(n\Cdot x),F(n\Cdot y)\}
$$
for any $x,y\in X$.
\end{remark}

\begin{lemma}\label{lemma:ee}
Let $F\colon X^n\to X$ be a quasitrivial operation and let $e\in X$. If $(n\Cdot e)$ is isolated for $F$, then $e$ is a neutral element. The converse holds if $n=2$.
\end{lemma}

\begin{proof}
We proceed by contradiction. If $e$ is not a neutral element, then there exist $i\in [n]$ and $u\in X\setminus\{e\}$ such that
$$
F((i-1)\Cdot e,u,(n-i)\Cdot e)~=~ e ~=~ F(n\Cdot e).
$$
It follows that $(n\Cdot e)$ is not isolated for $F$, which is a contradiction. The case $n=2$ was proved in \cite{CouDevMar}.
\end{proof}

\begin{example}
The following example shows that the converse of Lemma~\ref{lemma:ee} does not hold when $n\geq 3$. Let $X=\{a,b,e\}$ and let $F\colon X^3\to X$ be defined as
$$
F(x,y,z) ~=~
\begin{cases}
a, & \text{if there are more $a$'s than $b$'s among $x,y,z$,}\\
b, & \text{if there are more $b$'s than $a$'s among $x,y,z$,}\\
e, & \text{otherwise}.
\end{cases}
$$
It is easy to see that this operation is quasitrivial and has $e$ as the neutral element. However, we have $F(e,e,e)=F(a,b,e)$ and hence the point $(e,e,e)$ is not isolated for $F$.
\end{example}

%---------------------------------------------------------------------------------------
\section{Associative operations}

In this section we provide a characterization of the $n$-ary operations on a given chain that are quasitrivial, symmetric, nondecreasing, and associative, and we show that these operations are derived from associative binary operations. We also examine the special case where these operations have neutral elements.

\begin{proposition}\label{prop:MarMayTor}
Let $X$ be a chain. If $G\colon X^2\to X$ is quasitrivial, symmetric, and nondecreasing, then it is associative.
\end{proposition}

\begin{proof}
This result was established in the special case where $X$ is the real unit interval $[0,1]$ in \cite[Proposition~2]{MarMayTor03}. The proof therein is purely algebraic and works for any nonempty chain $X$.
\end{proof}

\begin{theorem}\label{thm:main2}
Let $X$ be a chain and let $F\colon X^n\to X$ $(n\geq 3)$ be quasitrivial, symmetric, and nondecreasing. The following assertions are equivalent.
\begin{enumerate}
\item[$(i)$] $F$ is associative.
\item[$(ii)$] $F((n-1)\Cdot x,y) = F(x,(n-1)\Cdot y)$ for all $x,y\in X$.
\item[$(iii)$] There exists a quasitrivial and nondecreasing operation $G\colon X^2\to X$ such that
\begin{equation}\label{eq:FUmm}
F(x_1,\ldots,x_n) ~=~ \textstyle{G(\bigwedge_{i=1}^nx_i{\,},\bigvee_{i=1}^nx_i)},\qquad x_1,\ldots,x_n\in X.
\end{equation}
\end{enumerate}
Moreover, the operation $G$ is unique, symmetric, and associative in assertion $(iii)$.
\end{theorem}

\begin{proof}
$(iii)\Rightarrow (i)$. Let $x_1,\ldots,x_{2n-1}\in X$ and $k\in [n-1]$. We only need to prove that
\begin{eqnarray}
\lefteqn{F(x_1,\ldots,x_{k-1},F(x_k,\ldots,x_{k+n-1}),x_{k+n},\ldots,x_{2n-1})}\nonumber\\
&=& F(x_1,\ldots,x_k,F(x_{k+1},\ldots,x_{k+n}),x_{k+n+1},\ldots,x_{2n-1}).\label{eq:av2}
\end{eqnarray}
Set $a=\bigwedge_{i=1}^{2n-1}x_i$ and $b=\bigvee_{i=1}^{2n-1}x_i$. By quasitriviality of $G$ we have $G(a,b)\in\{a,b\}$. Let us assume that $G(a,b)=a$ (the other case can be dealt with dually). We observe that if $a\in\{x_{\ell_1},\ldots,x_{\ell_n}\}$ for some pairwise distinct $\ell_1,\ldots,\ell_n\in [2n-1]$, then $F(x_{\ell_1},\ldots,x_{\ell_n})=a$. Indeed, by quasitriviality and nondecreasing monotonicity of $G$, we have $a=G(a,a)\leq G(a,\bigvee_{i=1}^nx_{\ell_i})\leq G(a,b)=a$, where $G(a,\bigvee_{i=1}^nx_{\ell_i})=F(x_{\ell_1},\ldots,x_{\ell_n})$. Now, combining this observation with the fact that $a\in\{x_1,\ldots,x_{2n-1}\}$ we immediately see that both sides of Eq.~\eqref{eq:av2} reduce to $a$.

$(i)\Rightarrow (ii)$. Suppose that $(i)$ holds and $(ii)$ does not. Then by quasitriviality there exist $x,y\in X$ with $x\neq y$ such that
\begin{itemize}
\item either $F((n-1)\Cdot x,y)=x$ and $F(x,(n-1)\Cdot y)=y$,
\item or $F((n-1)\Cdot x,y)=y$ and $F(x,(n-1)\Cdot y)=x$.
\end{itemize}
The second case cannot hold. Indeed, assuming for instance that $x<y$, by nondecreasing monotonicity we would have $x < y = F((n-1)\Cdot x,y) \leq  F(x,(n-1)\Cdot y) = x$, a contradiction. Thus we must have  $F((n-1)\Cdot x,y)=x$ and $F(x,(n-1)\Cdot y)=y$. We then have
$$
F(F(n\Cdot x),(n-1)\Cdot y) ~=~ y ~\neq ~ x ~=~ F((n-1)\Cdot x,F(x,(n-1)\Cdot y)),
$$
which shows that $F$ is not associative, a contradiction.

$(ii)\Rightarrow (iii)$. Define $G\colon X^2\to X$ as 
$$
G(x,y)=F((n-1)\Cdot x,y)=F(x,(n-1)\Cdot y).
$$ By definition, $G$ is quasitrivial, symmetric, and nondecreasing. It is also associative by Proposition~\ref{prop:MarMayTor}. Now, let $x_1,\ldots,x_n\in X$ and let $a=\bigwedge_{i=1}^n x_i$ and $b=\bigvee_{i=1}^nx_i$. By symmetry and nondecreasing monotonicity of $F$, we have
$$
G(a,b) ~=~ F((n-1)\Cdot a,b) ~\leq ~ F(x_1,\ldots,x_n) ~\leq ~ F(a,(n-1)\Cdot b) ~=~ G(a,b),
$$
which proves \eqref{eq:FUmm}. We then observe that $G$ is necessarily unique.
\end{proof}

\begin{remark}
\begin{enumerate}
\item[(a)] Theorem~\ref{thm:main2} also holds for $n=2$ but brings no information in this case. However, Proposition~\ref{prop:MarMayTor} shows that associativity follows from quasitriviality, symmetry, and nondecreasing monotonicity.
\item[(b)] Each of the conditions of Theorem~\ref{thm:main2} is necessary as soon as $n\geq 3$. Indeed,
\begin{itemize}
\item the sum operation $F(x_1,\ldots,x_n)=\sum_{i=1}^nx_i$ over the reals $X=\R$ is symmetric, nondecreasing, and associative, but not quasitrivial;
\item the projection operation $F(x_1,\ldots,x_n)=x_1$ over any chain $X$ is quasitrivial, nondecreasing, and associative, but not symmetric;
\item the ternary operation $F\colon L_3^2\to L_3$, with $L_3=\{1,2,3\}$, defined as $F(1,1,1)=1$, $F(x,y,z)=2$, if $2\in\{x,y,z\}$, $F(x,y,z)=3$ if $3\in\{x,y,z\}$ and $2\notin\{x,y,z\}$ is quasitrivial, symmetric, and associative, but it is not nondecreasing;
\item the ternary median operation, defined on any chain by
$$
\mathrm{median}(x,y,z) ~=~ (x\wedge y)\vee (y\wedge z)\vee (z\wedge x),
$$
is quasitrivial, symmetric, and nondecreasing, but it is not associative.
\end{itemize}
None of the four operations above is of the form given in \eqref{eq:FUmm}.
\item[(c)] In the proof of Theorem~\ref{thm:main2} the symmetry condition is used only to prove that $(ii)\Rightarrow (iii)$. We observe that this condition can then be relaxed into the following one:
$$
F((i-1)\Cdot x,y,(n-i)\Cdot x) ~=~ F((j-1)\Cdot x,y,(n-j)\Cdot x),\quad i,j\in [n], ~x,y\in X.
$$
\end{enumerate}
\end{remark}

\begin{definition}[{see \cite{Ack,DudMuk06}}]\label{de:FH5}
Let $F\colon X^n\to X$ and $H\colon X^2\to X$ be associative operations. $F$ is said to be \emph{derived from} (or \emph{reducible to}) $H$ if $F(x_1,\ldots,x_n)=x_1\circ\cdots\circ x_n$ for all $x_1,\ldots,x_n\in X$, where $x\circ y=H(x,y)$.
\end{definition}

\begin{corollary}\label{cor:fs5fs}
Let $X$ be a chain. If $F\colon X^n\to X$ is of the form \eqref{eq:FUmm}, where $G\colon X^2\to X$ is quasitrivial, symmetric, and nondecreasing, then $F$ is associative and derived from $G$.
\end{corollary}

\begin{proof}
Clearly, $F$ is quasitrivial, symmetric, and nondecreasing. The fact that it is also associative follows from Theorem~\ref{thm:main2}. To see that $F$ is derived from $G$, let $x_1,\ldots,x_n\in X$ and set $a=\bigwedge_{i=1}^nx_i$ and $b=\bigvee_{i=1}^nx_i$. By quasitriviality of $G$ we have $G(a,b)\in\{a,b\}$. Let us assume that $G(a,b)=a$ (the other case can be dealt with dually). For every $i\in [n]$, we have $G(a,x_i)=a$. Indeed, we have $a=G(a,a)\leq G(a,x_i)\leq G(a,b)=a$. Now, writing $G(x,y)=x\circ y$, we have $x_1\circ\cdots\circ x_n = a = G(a,b) = F(x_1,\ldots,x_n)$, which completes the proof.
\end{proof}

\begin{remark}\label{rem:2e}
The operation $H$ in Definition~\ref{de:FH5} need not be unique. For instance, if $X=\{a,b,c\}$, then the constant ternary operation $F\colon X^3\to X$ defined as $F=a$ is derived from the associative operation $H\colon X^2\to X$ defined as $H(c,c)=b$ and $H(x,y)=a$ for any $(x,y)\neq (c,c)$. However, $F$ is also derived from the constant operation $H'=a$. To give a second example, just consider the associative operations $F(x_1,x_2,x_3)\equiv x_1+x_2+x_3~(\mathrm{mod}~2)$, $H(x_1,x_2)\equiv x_1+x_2~(\mathrm{mod}~2)$, and  $H'(x_1,x_2)\equiv x_1+x_2+1~(\mathrm{mod}~2)$ on $X=\mathbb{Z}_2$. It is clear that $F$ is derived from both $H$ and $H'$. Interestingly, we also observe that $F$ is quasitrivial but neither $H$ nor $H'$ is quasitrivial.
\end{remark}

\begin{proposition}
Assume that the operation $F\colon X^n\to X$ is associative and derived from associative operations $H\colon X^2\to X$ and $H'\colon X^2\to X$. If $H$ and $H'$ are idempotent (resp.\ have the same neutral element), then $H=H'$.
\end{proposition}

\begin{proof}
Assume that $H$ and $H'$ are idempotent (the other case can be dealt with similarly) and let us write $H(x,y)=x\circ y$ and $H'(x,y)=x\diamond y$. Then, for any $x,y\in X$ we have
$$
x\circ y ~=~ \underbrace{x\circ\cdots\circ x}_{n-1}\circ{\,} y ~=~ F((n-1)\Cdot x,y) ~=~ \underbrace{x\diamond\cdots\diamond x}_{n-1}\diamond{\,} y ~=~ x\diamond y,
$$
which shows that $H=H'$.
\end{proof}

The following proposition provides a characterization of the class of quasitrivial, symmetric, associative operations $F\colon X^n\to X$ that are derived from quasitrivial and associative binary operations. This result was previously known and established by using algebraic arguments (see, e.g., \cite[Corollary 4.10]{Ack} and the references therein). For the sake of self-containment we provide a very elementary proof.

We first consider a lemma.

\begin{lemma}\label{lemma:ack2}
Assume that the operation $F\colon X^n\to X$ is symmetric, associative, and derived from a surjective (i.e., onto) and associative operation $H\colon X^2\to X$. Then $H$ is symmetric.
\end{lemma}

\begin{proof}
Let us write $H(x,y)=x\circ y$ for all $x,y\in X$. For any $x,y\in X$ there exist $y_1,\ldots,y_{n-2}\in X$ and $z_1,\ldots,z_{n-2}\in X$ such that
\begin{eqnarray*}
H(x,y) &=& x\circ y ~=~ x\circ y_1\circ z_1 ~=~ x\circ y_1\circ y_2\circ z_2 ~=~ \cdots\\
&=& x\circ y_1\circ\cdots\circ y_{n-2}\circ z_{n-2} ~=~ F(x,y_1,\ldots,y_{n-2},z_{n-2})\\
&=& F(y_1,\ldots,y_{n-2},z_{n-2},x) ~=~ \cdots ~=~ y\circ x ~=~ H(y,x),
\end{eqnarray*}
which shows that $H$ is symmetric.
\end{proof}

\begin{proposition}\label{prop:ack}
A symmetric associative operation $F\colon X^n\to X$ is derived from a quasitrivial and associative operation $H\colon X^2\to X$ if and only if there exists a linear ordering $\preceq$ on $X$ such that $F$ is the maximum operation on $(X,\preceq)$, i.e.,
\begin{equation}\label{eq:conjugMax}
F(x_1,\ldots,x_n) ~=~ x_1\vee_{\preceq}\cdots\vee_{\preceq} x_n{\,},\qquad x_1,\ldots,x_n\in X.
\end{equation}
\end{proposition}

\begin{proof}
(Sufficiency) Trivial.

(Necessity)
Since $H$ is quasitrivial, it is idempotent and hence it is also symmetric by Lemma~\ref{lemma:ack2}. Using quasitriviality and associativity of $H$, it is then easy to see that the binary relation $\preceq$ on $X$ defined as
\begin{equation}\label{eq:defLor4}
x ~ \preceq ~ y\quad\Leftrightarrow\quad H(x,y)~=~y{\,},\qquad x,y\in X,
\end{equation}
is a linear order. Using symmetry, we then have $H(x,y)=x\vee_{\preceq}y$ for all $x,y\in X$.
\end{proof}

\begin{remark}
The second example given in Remark~\ref{rem:2e} shows that the quasitriviality of $H$ is necessary in Proposition~\ref{prop:ack}.
\end{remark}

For arbitrary chains $(X,\leq)$ and $(X,\preceq)$, the operations $F\colon X^n\to X$ of the form \eqref{eq:conjugMax} need not be nondecreasing for $\leq$. In the following proposition we characterize the class of linear orderings $\preceq$ on $X$ for which those operations $F$ are nondecreasing for $\leq$.

\begin{definition}\label{de:Black}
Let $(X,\leq)$ and $(X,\preceq)$ be chains. We say that the linear ordering $\preceq$ is \emph{single-peaked for $\leq$} if for any $a,b,c\in X$ such that $a<b<c$ we have $b\prec a$ or $b\prec c$.
\end{definition}

\begin{remark}\label{rem:Bp4}
The concept of single-peaked linear ordering was introduced for finite chains in social choice theory by Black~\cite{Bla48,Bla87} (see \cite{BerPer06,Fou01} for more recent references). Here we simply apply the same definition to arbitrary chains. Thus defined, $\preceq$ is single-peaked for $\leq$ if and only if, from among three distinct elements of $X$, the centrist one (for $\leq$) is never ranked last by $\preceq$.
\end{remark}

\begin{proposition}\label{prop:Bl56}
Let $(X,\leq)$ and $(X,\preceq)$ be chains and let $F\colon X^n\to X$ be of the form \eqref{eq:conjugMax}. Then $F$ is nondecreasing for $\leq$ if and only if $\preceq$ is single-peaked for $\leq$.
\end{proposition}

\begin{proof}
(Necessity) We proceed by contradiction. Suppose that there exist $a,b,c\in X$ satisfying $a<b<c$ such that $b\succ a$ and $b\succ c$. Then by nondecreasing monotonicity we clearly have
$$
b ~=~ F(a,(n-1)\Cdot b) ~\leq ~ F(a,(n-1)\Cdot c) ~\leq ~ F(b,(n-1)\Cdot c) ~=~ b
$$
and hence $F(a,(n-1)\Cdot c)=b$, which contradicts quasitriviality.

(Sufficiency) We proceed by contradiction. Suppose that there exist $k\in [n]$ and $(x_1,\ldots,x_n),(x'_1,\ldots,x'_n)\in X^n$ such that $x_k<x'_k$ and $x_i=x'_i$ for all $i\in [n]\setminus\{k\}$ and $F(x_1,\ldots,x_n)>F(x'_1,\ldots,x'_n)$, that is,
$$
x_j\vee_{\preceq} x_k ~ > ~ x_j\vee_{\preceq} x'_k{\,},
$$
where $x_j=x_1\vee_{\preceq}\cdots\vee_{\preceq} x_{k-1}\vee_{\preceq} x_{k+1}\vee_{\preceq}\cdots\vee_{\preceq} x_n$.

We only have two relevant cases to consider.
\begin{itemize}
\item If $x_k\preceq x_j\preceq x'_k$, then we obtain $x_k<x'_k<x_j$.
\item If $x'_k\preceq x_j\preceq x_k$, then we obtain $x_j<x_k<x'_k$.
\end{itemize}
We immediately reach a contradiction since $\preceq$ cannot be single-peaked for $\leq$ in each of these two cases.
\end{proof}

The following two propositions provide characterizations of single-peaked\-ness. Recall first that, for any chain $(X,\leq)$, a subset $C$ of $X$ is said to be \emph{convex for $\leq$} if for any $a,b,c\in X$ such that $a<b<c$ we have that $a,c\in C$ implies $b\in C$.

\begin{proposition}
Let $(X,\leq)$ and $(X,\preceq)$ be chains. Then $\preceq$ is single-peaked for $\leq$ if and only if for every $t\in X$ the set $\{x\in X\mid x\preceq t\}$ is convex for $\leq$.
\end{proposition}

\begin{proof}
(Necessity) We proceed by contradiction. Suppose that there exist $a,b,c,t_0\in X$ satisfying $a<b<c$ such that $a,c\in\{x\in X\mid x\preceq t_0\}$ and $b\notin\{x\in X\mid x\preceq t_0\}$. This means that $a\preceq t_0\prec b$ and $c\preceq t_0\prec b$, a contradiction.

(Sufficiency) We proceed by contradiction. Suppose that there exist $a,b,c\in X$ satisfying $a<b<c$ such that $b\succ a$ and $b\succ c$. Setting $t_0=a\vee_{\preceq} c$, we have $a,c\in\{x\in X\mid x\preceq t_0\}$. By convexity we also have $b\in\{x\in X\mid x\preceq t_0\}$. Therefore $a\vee_{\preceq} c\prec b\preceq t_0=a\vee_{\preceq} c$, a contradiction.
\end{proof}

\begin{proposition}\label{prop:sisd}
Let $(X,\leq)$ and $(X,\preceq)$ be chains. Then $\preceq$ is single-peaked for $\leq$ if and only if for any $x_0,x_1,x_2\in X$ such that $x_0\prec x_1$ and $x_0\prec x_2$ we have
\begin{equation}\label{eq:imp}
x_0<x_1<x_2\quad\text{or}\quad x_2<x_1<x_0\quad\Rightarrow\quad x_1\prec x_2.
\end{equation}
\end{proposition}

\begin{proof}
(Necessity) We proceed by contradiction. Suppose that there exist $x_0,x_1,x_2\in X$ satisfying $x_0\prec x_1$ and $x_0\prec x_2$ for which \eqref{eq:imp} fails to hold, i.e., either ($x_0<x_1<x_2$ and $x_0\prec x_2\prec x_1$) or ($x_2<x_1<x_0$ and $x_0\prec x_2\prec x_1$), which clearly shows that $\preceq$ is not single-peaked for $\leq$.

(Sufficiency) We proceed by contradiction. Suppose that there exist $a,b,c\in X$ satisfying $a<b<c$ such that $b\succ a$ and $b\succ c$. Letting $x_0=a\wedge_{\preceq} c$, $x_1=b$, and $x_2=a\vee_{\preceq} c$, we obtain either ($x_0<x_1<x_2$ and $x_0\prec x_2\prec x_1$) or ($x_2<x_1<x_0$ and $x_0\prec x_2\prec x_1$), which shows that \eqref{eq:imp} fails to hold, a contradiction.
\end{proof}

\begin{corollary}\label{cor:sisd22}
Let $(X,\leq)$ and $(X,\preceq)$ be chains and suppose that $(X,\preceq)$ has a minimal element $x_0$. Then $\preceq$ is single-peaked for $\leq$ if and only if \eqref{eq:imp} holds for all $x_1,x_2\in X$.
\end{corollary}

\begin{proof}
(Necessity) The result follows from Proposition~\ref{prop:sisd}.

(Sufficiency) We proceed by contradiction. Suppose that there exist $a,b,c\in X$ satisfying $a<b<c$ such that $b\succ a$ and $b\succ c$. Since $x_0$ is the minimal element of $(X,\preceq)$, we must have $x_0\neq b$. If $x_0<b$, then setting $x_1=b$ and $x_2=c$, we obtain $x_0<x_1<x_2$ and $x_2\prec x_1$, a contradiction. We arrive at a similar contradiction if $x_0>b$.
\end{proof}

Using Propositions~\ref{prop:MarMayTor}, \ref{prop:ack}, \ref{prop:Bl56}, Theorem~\ref{thm:main2}, and Corollary~\ref{cor:fs5fs}, we easily derive the following theorem.

\begin{theorem}\label{thm:f456dfs}
Let $X$ be a chain and let $F\colon X^n\to X$ be an operation. The following assertions are equivalent.
\begin{enumerate}
\item[(i)] $F$ is quasitrivial, symmetric, nondecreasing, and associative (associativity can be ignored when $n=2$).
\item[(ii)] $F$ is of the form \eqref{eq:FUmm}, where $G\colon X^2\to X$ is quasitrivial, symmetric, and nondecreasing.
\item[(iii)] $F$ is of the form \eqref{eq:conjugMax} for some linear ordering $\preceq$ on $X$ that is single-peaked for $\leq$.
\end{enumerate}
If any of these assertions holds, then $G$ is associative and $F$ is derived from $G$.
\end{theorem}

We now investigate the particular case where the operations satisfying any of the assertions of Theorem~\ref{thm:main2} have neutral elements.

\begin{proposition}\label{prop:eee}
Let $X$ be a chain and let $e\in X$. Assume that $F\colon X^n\to X$ is of the form \eqref{eq:FUmm}, where $G\colon X^2\to X$ is quasitrivial and symmetric. Then the following assertions are equivalent.
\begin{enumerate}
\item[(i)] $e$ is a neutral element of $F$.
\item[(ii)] $e$ is a neutral element of $G$.
\item[(iii)] The point $(e,e)$ is isolated for $G$.
\item[(iv)] The point $(n\Cdot e)$ is isolated for $F$.
\end{enumerate}
In particular, if $F$ has a neutral element, then it is unique.
\end{proposition}

\begin{proof}
The equivalence (i) $\Leftrightarrow$ (ii) is straightforward. The equivalence (ii) $\Leftrightarrow$ (iii) and the implication (iv) $\Rightarrow$ (i) follow from Lemma~\ref{lemma:ee}. Let us now prove the implication (iii) $\Rightarrow$ (iv). Suppose that $(e,e)$ is isolated for $G$ and that $(n\Cdot e)$ is not isolated for $F$. Then, there exists $(x_1,\ldots,x_n)\neq (n\Cdot e)$ such that $G(e,e)=F(n\Cdot e)=F(x_1,\ldots,x_n)=G(\bigwedge_{i=1}^nx_i,\bigvee_{i=1}^nx_i)$, which contradicts the assumption that $(e,e)$ is isolated for $G$.
\end{proof}

\begin{remark}\label{rem:Ne6}
We observe that if an operation $F\colon X^n\to X$ is of the form \eqref{eq:conjugMax} for some linear ordering $\preceq$ on $X$, then $F$ has a neutral element $e\in X$ if and only if $(X,\preceq)$ has a minimal element $x_0$. In this case $e=x_0$.
\end{remark}

\begin{example}
Consider the real operation $F\colon [0,1]^n\to [0,1]$ of the form \eqref{eq:FUmm}, where $G\colon [0,1]^2\to [0,1]$ is defined as $G(x_1,x_2)=x_1\vee x_2$, if $x_1,x_2>0$, and $G(x_1,x_2)=0$, otherwise. It is easy to see that $F$ satisfies the conditions of assertion (ii) of Theorem~\ref{thm:f456dfs} and that $G$ does not have any neutral element. It follows that $F$ has no neutral element either by Proposition~\ref{prop:eee}. By Proposition~\ref{prop:ack} and \eqref{eq:defLor4} we see that $F$ is of the form \eqref{eq:conjugMax}, where $\preceq$ is the single-peaked linear ordering on $[0,1]$ defined as
$$
x ~ \preceq ~ y \quad\Leftrightarrow\quad 0<x\leq y\quad\text{or}\quad y=0{\,},\qquad x,y\in [0,1].
$$
From this representation it is easy to see that $F$ can also be expressed as
$$
F(x_1,\ldots,x_n) ~=~ \textstyle{f(\bigvee_{i=1}^nf^{-1}(x_i))},\qquad x_1,\ldots,x_n\in [0,1],
$$
where $f\colon\left]0,1\right]\cup\{2\}\to [0,1]$ is defined as $f(y)=y$, if $y\in\left]0,1\right]$, and $f(2)=0$.
\end{example}

\begin{example}
Consider the real operation $F\colon [0,1]^2\to [0,1]$ defined as $F(x,y)=x\vee y$, if $x+y\geq 1$, and $F(x,y)=x\wedge y$, otherwise. It is easy to see that $F$ satisfies the conditions of assertion (ii) of Theorem~\ref{thm:f456dfs} (with $G=F$) and that it has the neutral element $e=\frac{1}{2}$. By Proposition~\ref{prop:ack} and \eqref{eq:defLor4} we see that $F$ is of the form \eqref{eq:conjugMax}, where $\preceq$ is the single-peaked linear ordering on $[0,1]$ defined as
$$
x ~ \preceq ~ y \quad\Leftrightarrow\quad (y\leq x<1-y~\text{ or }~1-y\leq x\leq y){\,},\qquad x,y\in [0,1].
$$
Interestingly, we observe that for every $x\in [0,1]$, there is no $y\in [0,1]$ such that $x\prec y\prec 1-x$ or $1-x\prec y\prec x$. From this observation we can show that the chain $([0,1],\preceq)$ cannot be embedded into the real chain $(\R,\leq)$. Indeed, suppose on the contrary that there exits an injective (i.e., one-to-one) map $g\colon ([0,1],\preceq)\to (\R,\leq)$ such that $x\preceq y$ if and only if $g(x)\leq g(y)$ for all $x,y\in [0,1]$. For any $x\in [0,1]\setminus\{\frac{1}{2}\}$ we have $x\neq 1-x$ and hence $g(x)\neq g(1-x)$. Let $I_x$ denote the real open interval having $g(x)$ and $g(1-x)$ as endpoints. We can always pick a rational number $q_x$ in $I_x$. For any distinct $x,y\in [0,1]\setminus\{\frac{1}{2}\}$ we have $I_x\cap I_y=\varnothing$ and hence $q_x\neq q_y$. It follows that the set $\{q_x\mid x\in [0,1]\setminus\{\frac{1}{2}\}\}$ is uncountable, a contradiction.
\end{example}

Recall that a \emph{uninorm} on a chain $X$ is a binary operation $U\colon X^2\to X$ that is associative, symmetric, nondecreasing, and has a neutral element (see \cite{DeB99,YagRyb96}). It is noteworthy that any idempotent uninorm is quasitrivial. This fact can be easily verified by following the first few steps of the proof of \cite[Theorem~3]{CzoDre84}, which actually hold on any chain.

The concept of uninorm can be easily extended to $n$-ary operations as follows.

\begin{definition}[{see \cite{KisSom}}]
Let $X$ be a chain. An \emph{$n$-ary uninorm} is an operation $F\colon X^n\to X$ that is associative, symmetric, nondecreasing, and has a neutral element.
\end{definition}

\begin{lemma}[{see \cite{DudMuk06}}]\label{lemma:Dud3}
If $F\colon X^n\to X$ is associative and has a neutral element $e\in X$, then $F$ is derived from the associative operation $H\colon X^2\to X$ defined as $H(x,y)=F(x,(n-2)\Cdot e,y)$.
\end{lemma}

\begin{corollary}\label{cor:fs6}
Let $X$ be a chain. Any idempotent $n$-ary uninorm $F\colon X^n\to X$ satisfies the assertions of Theorem~\ref{thm:f456dfs}. In particular,
\begin{itemize}
\item $F$ has a unique neutral element $e$;
\item $F((n-1)\Cdot x,y) = F(x,(n-1)\Cdot y) = F(x,(n-2)\Cdot e,y)$ for all $x,y\in X$;
\item the operation $G$ in assertion (ii) is an idempotent uninorm with $e$ as the neutral element;
\item the chain $(X,\preceq)$ in assertion (iii) has the minimal element $e$.
\end{itemize}

\end{corollary}

\begin{proof}
By Lemma~\ref{lemma:Dud3}, $F$ is derived from the associative operation $H\colon X^2\to X$ defined as $H(x,y)=x\circ y=F(x,(n-2)\Cdot e,y)$. By definition, $H$ is associative, symmetric, nondecreasing, and has a neutral element. We now show that it is also idempotent. Although this property immediately follows from \cite[Lemma~3.5]{KisSom} we present here an alternative and very simple proof of it. Suppose that $H(x,x)=z\neq x$. If $x<z$ (the case $x>z$ is similar), then by nondecreasing monotonicity of $H$ we have
\begin{eqnarray*}
x & < & z ~=~ x\circ x ~\leq ~ z\circ x~=~ x\circ x\circ x ~\leq ~ z\circ x\circ x ~=~ x\circ x\circ x\circ x ~\leq ~ \cdots\\
&=& \underbrace{x\circ\cdots\circ x}_n ~=~ F(x,\ldots,x) ~=~ x,
\end{eqnarray*}
a contradiction. Therefore $H$ is an idempotent uninorm and hence it is quasitrivial (as observed above). It follows that $F$ is quasitrivial and hence satisfies assertion (i) of Theorem~\ref{thm:f456dfs}. Also, we have
$$
F(x,(n-2)\Cdot e,y) ~=~ x\circ y ~=~ \underbrace{x\circ\cdots\circ x}_{n-1}\circ{\,} y ~=~ x\circ\underbrace{y\circ\cdots\circ y}_{n-1}.
$$
The rest of the corollary follows from Proposition~\ref{prop:eee} and Remark~\ref{rem:Ne6}.
\end{proof}

\begin{corollary}\label{cor:fs62}
Let $X$ be a chain and let $F\colon X^n\to X$ be an operation. Then $F$ is an idempotent $n$-ary uninorm if and only if there exists an idempotent uninorm $U\colon X^2\to X$ such that
$$
F(x_1,\ldots,x_n) ~=~ \textstyle{U(\bigwedge_{i=1}^nx_i{\,},\bigvee_{i=1}^nx_i)},\qquad x_1,\ldots,x_n\in X.
$$
In this case, the uninorm $U$ is uniquely defined as $U(x,y)=F((n-1)\Cdot x,y)$.
\end{corollary}

\begin{remark}
The results presented in this section strongly rely on the symmetry of the operations $F\colon X^n\to X$. The generalization of these results to the nonsymmetric case is a topic of ongoing research. On this issue, partial results can be found, e.g., in \cite[Lemma 3.15]{KisSom}.
\end{remark}

%---------------------------------------------------------------------------------------
\section{Bisymmetric operations}

In this section we investigate bisymmetric $n$-ary operations and derive a few equivalences involving associativity and bisymmetry. For instance we show that if an $n$-ary operation has a neutral element, then it is bisymmetric if and only if it is associative and symmetric. Also, if an $n$-ary operation is quasitrivial and symmetric, then it is associative if and only if it is bisymmetric. In particular this latter observation enables us to replace associativity with bisymmetry in Theorems~\ref{thm:main2}, \ref{thm:f456dfs}, and Corollary~\ref{cor:fs5fs}.

\begin{lemma}[{see \cite{CouDevMar}}]\label{lemma:bis}
Let $F\colon X^2\to X$ be an operation. Then the following assertions hold.
\begin{enumerate}
\item[(a)] If $F$ is bisymmetric and has a neutral element, then it is associative and symmetric.
\item[(b)] If $F$ is associative and symmetric, then it is bisymmetric.
\item[(c)] If $F$ is bisymmetric and quasitrivial, then it is associative.
\end{enumerate}
\end{lemma}

\begin{definition}
We say that a function $F\colon X^n\to X$ is \emph{ultrabisymmetric} if
$$
F(F(\mathbf{r}_1),\ldots,F(\mathbf{r}_n)) ~=~ F(F(\mathbf{r}'_1),\ldots,F(\mathbf{r}'_n))
$$
for all $n\times n$ matrices $[\mathbf{r}_1 ~\cdots ~\mathbf{r}_n]^T,[\mathbf{r}'_1 ~\cdots ~\mathbf{r}'_n]^T\in X^{n\times n}$, where $[\mathbf{r}'_1 ~\cdots ~\mathbf{r}'_n]^T$ is obtained from $[\mathbf{r}_1 ~\cdots ~\mathbf{r}_n]^T$ by exchanging two entries only.
\end{definition}

\begin{proposition}\label{prop:ds5}
Let $F\colon X^n\to X$ be an operation. If $F$ is ultrabisymmetric, then it is bisymmetric. The converse holds whenever $F$ is symmetric.
\end{proposition}

\begin{proof}
We immediately see that any ultrabisymmetric operation is bisymmetric (just apply ultrabisymmetry repeatedly to exchange the $(i,j)$- and $(j,i)$-entries for all $i,j\in [n]$).

Now suppose that $F\colon X^n\to X$ is symmetric and bisymmetric. Then we have
$$
F(F(\mathbf{r}_1),\ldots,F(\mathbf{r}_n)) ~=~ F(F(\mathbf{r}'_1),\ldots,F(\mathbf{r}'_n))
$$
for all matrices $[\mathbf{r}_1 ~\cdots ~\mathbf{r}_n]^T,[\mathbf{r}'_1 ~\cdots ~\mathbf{r}'_n]^T\in X^{n\times n}$, where $[\mathbf{r}'_1 ~\cdots ~\mathbf{r}'_n]^T$ is obtained from $[\mathbf{r}_1 ~\cdots ~\mathbf{r}_n]^T$ by permuting the entries of any column or any row. By applying three times this property, we can easily exchange two arbitrary entries of the matrix. Indeed, exchanging the $(i,j)$- and $(k,l)$-entries can be performed through the following three steps: exchange the $(i,j)$- and $(i,l)$-entries in row $i$, exchange the $(i,l)$- and $(k,l)$-entries in column $l$, and exchange the $(i,j)$- and $(i,l)$-entries in row $i$.
\end{proof}

\begin{remark}
\begin{enumerate}
\item[(a)] The symmetry property is necessary in Proposition~\ref{prop:ds5}. Indeed, for any $k\in [n]$, the $k$th projection operation $F\colon X^n\to X$ defined as $F(x_1,\ldots,x_n)=x_k$ is bisymmetric but not ultrabisymmetric.
\item[(b)] An ultrabisymmetric operation need not be symmetric. For instance, consider the operation $F\colon X^2\to X$, where $X=\{a,b,c\}$, defined by $F(a,c)=a$ and $F(x,y)=b$ for every $(x,y)\neq (a,c)$. Clearly, this operation is not symmetric. However, it is ultrabisymmetric since $F(F(x,y),F(u,v))=b$ for all $x,y,u,v\in X$.
\end{enumerate}
\end{remark}

\begin{lemma}\label{lemma:surj65}
If $F\colon X^n\to X$ is surjective (i.e., onto) and ultrabisymmetric, then it is symmetric.
\end{lemma}

\begin{proof}
Let $x_1,\ldots,x_n\in X$. Then there exists a matrix $[\mathbf{r}_1 ~\cdots ~\mathbf{r}_n]^T\in X^{n\times n}$ such that $x_i=F(\mathbf{r}_i)$ for $i=1,\ldots,n$. By ultrabisymmetry,
$$
F(x_1,\ldots,x_n) ~=~ F(F(\mathbf{r}_1),\ldots,F(\mathbf{r}_n))
$$
is symmetric in $x_1,\ldots,x_n$.
\end{proof}

\begin{remark}\label{rem:surj65}
We observe that if $F\colon X^n\to X$ is idempotent or quasitrivial or has a neutral element, then it is surjective.
\end{remark}

\begin{lemma}\label{lemma:cons65}
If $F\colon X^n\to X$ is quasitrivial, then for any $x,y\in X$, there exists $k\in [n]$ such that
$$
F((k-1)\Cdot x,(n-k+1)\Cdot y)~=~y\quad\text{and}\quad F(k\Cdot x,(n-k)\Cdot y)~=~x.
$$
\end{lemma}

\begin{proof}
We proceed by contradiction. Suppose that there exist $x,y\in X$, with $x\neq y$, such that for every $k\in [n]$ we have
\begin{equation}\label{eq:co4t}
F((k-1)\Cdot x,(n-k+1)\Cdot y)~=~x\quad\text{or}\quad F(k\Cdot x,(n-k)\Cdot y)~=~y.
\end{equation}
Using the fact that $F(n\Cdot y)=y$ we see that only the second condition of \eqref{eq:co4t} holds. When $k=n$ this gives $F(n\Cdot x)=y$, a contradiction.
\end{proof}

\begin{proposition}\label{prop:19gz}
If $F\colon X^n\to X$ is quasitrivial and ultrabisymmetric, then it is associative and symmetric.
\end{proposition}

\begin{proof}
Symmetry immediately follows from Lemma~\ref{lemma:surj65} and Remark~\ref{rem:surj65}. Let us prove that associativity holds. Let $x_1,\ldots,x_{2n-1}\in X$ and let $i\in [n-1]$. By Lemma~\ref{lemma:cons65} there exists $k\in [n]$ such that
$$
F((k-1)\Cdot x_i,(n-k+1)\Cdot x_{i+n})~=~x_{i+n}\quad\text{and}\quad F(k\Cdot x_i,(n-k)\Cdot x_{i+n})~=~x_i.
$$
We then have
\begin{multline*}
F(x_1,\ldots,x_{i-1},F(x_i,\ldots,x_{i+n-1}),x_{i+n},\ldots,x_{2n-1})\\
=~ F(x_1,\ldots,x_{i-1},F(x_i,\ldots,x_{i+n-1}),F((k-1)\Cdot x_i,(n-k+1)\Cdot x_{i+n}),x_{i+n+1},\ldots,x_{2n-1})
\end{multline*}
Replacing $x_j$ with $F(n\Cdot x_j)$ for all $j\in [2n-1]\setminus\{i,\ldots,i+n\}$ and then applying ultrabisymmetry repeatedly to exchange the $(n-1)$-tuples
$$
(x_{i+1},\ldots,x_{i+n-1})\quad\text{and}\quad((k-1)\Cdot x_i,(n-k)\Cdot x_{i+n}),
$$
we see that the latter expression becomes
\begin{multline*}
F(x_1,\ldots,x_{i-1},F(k\Cdot x_i,(n-k)\Cdot x_{i+n}),F(x_{i+1},\ldots,x_{i+n}),x_{i+n+1},\ldots,x_{2n-1})\\
=~ F(x_1,\ldots,x_i,F(x_{i+1},\ldots,x_{i+n}),x_{i+n+1},\ldots,x_{2n-1}).
\end{multline*}
This shows that $F$ is associative.
\end{proof}

\begin{remark}
Ultrabisymmetry cannot be relaxed into bisymmetry in Proposition~\ref{prop:19gz}. For instance, the ternary operation $F(x,y,z)=y$ is quasitrivial and bisymmetric, but it is neither associative nor symmetric. This example also shows that the result stated in Lemma~\ref{lemma:bis}(c) cannot be extended to $n$-ary operations.
\end{remark}

\begin{proposition}\label{prop:21ft}
If $F\colon X^n\to X$ is associative and symmetric, then it is ultrabisymmetric.
\end{proposition}

\begin{proof}
Let $[\mathbf{r}_1 ~\cdots ~\mathbf{r}_n]^T,[\mathbf{r}'_1 ~\cdots ~\mathbf{r}'_n]^T\in X^{n\times n}$, where $[\mathbf{r}'_1 ~\cdots ~\mathbf{r}'_n]^T$ is obtained from $[\mathbf{r}_1 ~\cdots ~\mathbf{r}_n]^T$ by exchanging the $(i,j)$- and $(k,l)$-entries for some $i,j,k,l\in [n]$. We only need to prove that
$$
F(F(\mathbf{r}_1),\ldots,F(\mathbf{r}_n)) ~=~ F(F(\mathbf{r}'_1),\ldots,F(\mathbf{r}'_n)).
$$
Permuting the rows of $[\mathbf{r}_1 ~\cdots ~\mathbf{r}_n]^T$ if necessary (this is allowed by symmetry), we may assume that $k=i+1$. Denote by $x_{i,j}$ (resp.\ $x_{k,l}$) the $(i,j)$-entry (resp.\ $(k,l)$-entry) of $[\mathbf{r}_1 ~\cdots ~\mathbf{r}_n]^T$.

Using associativity and symmetry, we see that there exist $p,q\in\{1,\ldots,n\}$, with $p\neq j$ and $q\neq l$, such that
\begin{multline*}
F(F(\mathbf{r}_1),\ldots,F(\mathbf{r}_n))\\
=~ F(F(\mathbf{r}_1),\ldots,F(\mathbf{r}_{i-1}),F(x_{i,p},\ldots,x_{i,j}),F(x_{k,l},\ldots,x_{k,q}),F(\mathbf{r}_{k+1}),\ldots,F(\mathbf{r}_n))\\
=~ F(F(\mathbf{r}_1),\ldots,F(\mathbf{r}_{i-1}),x_{i,p},F(\ldots,x_{i,j},F(x_{k,l},\ldots,x_{k,q})),F(\mathbf{r}_{k+1}),\ldots,F(\mathbf{r}_n))\\
=~ F(F(\mathbf{r}_1),\ldots,F(\mathbf{r}_{i-1}),x_{i,p},F(\ldots,F(x_{i,j},x_{k,l},\ldots),x_{k,q}),F(\mathbf{r}_{k+1}),\ldots,F(\mathbf{r}_n)).
\end{multline*}
This shows that $F$ is ultrabisymmetric since the latter expression is symmetric in $x_{i,j}$ and $x_{k,l}$.
\end{proof}

\begin{corollary}\label{cor:24f1}
If $F\colon X^n\to X$ is quasitrivial, then it is associative and symmetric if and only if it is ultrabisymmetric.
\end{corollary}

\begin{proof}
The statement immediately follows from Propositions~\ref{prop:19gz} and \ref{prop:21ft}.
\end{proof}

\begin{remark}
If $F\colon X^n\to X$ is ultrabisymmetric but not quasitrivial, then it need not be associative (e.g., $F(x,y,z)=2x+2y+2z$ when $X=\R$).
\end{remark}

\begin{corollary}\label{cor:24f}
If $F\colon X^n\to X$ is quasitrivial and symmetric, then it is associative if and only if it is bisymmetric.
\end{corollary}

\begin{proof}
The statement immediately follows from Propositions~\ref{prop:ds5}, \ref{prop:19gz}, and \ref{prop:21ft}.
\end{proof}

From Corollary~\ref{cor:24f} we immediately derive the following theorem, which is an important but surprising result.

\begin{theorem}
In Theorems~\ref{thm:main2}, \ref{thm:f456dfs}, and Corollary~\ref{cor:fs5fs} we can replace associativity with bisymmetry.
\end{theorem}

We end this section by investigating bisymmetric operations that have neutral elements.

\begin{proposition}\label{prop:20gt}
If $F\colon X^n\to X$ is bisymmetric and has a neutral element, then it is associative and symmetric.
\end{proposition}

\begin{proof}
Let $e$ be a neutral element of $F$. Let us first prove symmetry. Let $x_1,\ldots,x_n\in X$, let $i,j\in [n]$, and let $[\mathbf{c}_1 ~\cdots ~\mathbf{c}_n]=[\mathbf{r}_1 ~\cdots ~\mathbf{r}_n]^T\in X^{n\times n}$ be defined as
$$
\mathbf{r}_k ~=~
\begin{cases}
((j-1)\Cdot e,x_i,(n-j)\Cdot e), & \text{if $k=i$}\\
((i-1)\Cdot e,x_j,(n-i)\Cdot e), & \text{if $k=j$}\\
((k-1)\Cdot e,x_k,(n-k)\Cdot e), & \text{otherwise}.
\end{cases}
$$
By bisymmetry we have
\begin{eqnarray*}
F(x_1,\ldots,x_i,\ldots,x_j,\ldots,x_n) &=& F(F(\mathbf{r}_1),\ldots,F(\mathbf{r}_n)) ~=~ F(F(\mathbf{c}_1),\ldots,F(\mathbf{c}_n))\\
&=& F(x_1,\ldots,x_j,\ldots,x_i,\ldots,x_n).
\end{eqnarray*}
This shows that $F$ is symmetric.

Let us now show that $F$ is associative by using ultrabisymmetry (which follows from bisymmetry and symmetry by Proposition~\ref{prop:ds5}). Let $x_1,\ldots,x_{2n-1}\in X$, let $i\in [n-1]$ and let $[\mathbf{r}_1 ~\cdots ~\mathbf{r}_n]^T,[\mathbf{r}'_1 ~\cdots ~\mathbf{r}'_n]^T\in X^{n\times n}$ be defined as
$$
\mathbf{r}_k ~=~
\begin{cases}
(x_k,(n-1)\Cdot e), & \text{if $k<i$}\\
(x_i,\ldots,x_{i+n-1}), & \text{if $k=i$}\\
(x_{k+n-1},(n-1)\Cdot e), & \text{if $k>i$}
\end{cases}
$$
and
$$
\mathbf{r}'_k ~=~
\begin{cases}
(x_k,(n-1)\Cdot e), & \text{if $k<i+1$}\\
(x_{i+1},\ldots,x_{i+n}), & \text{if $k=i+1$}\\
(x_{k+n-1},(n-1)\Cdot e), & \text{if $k>i+1$}.
\end{cases}
$$
Using ultrabisymmetry, we then have
\begin{multline*}
F(x_1,\ldots,x_{i-1},F(x_i,\ldots,x_{i+n-1}),x_{i+n},\ldots,x_{2n-1})~=~ F(F(\mathbf{r}_1),\ldots,F(\mathbf{r}_n))\\
=~ F(F(\mathbf{r}'_1),\ldots,F(\mathbf{r}'_n))~=~  F(x_1,\ldots,x_i,F(x_{i+1},\ldots,x_{i+n}),x_{i+n+1},\ldots,x_{2n-1}).
\end{multline*}
This shows that $F$ is associative.
\end{proof}

\begin{corollary}
Assume that $F\colon X^n\to X$ has a neutral element. Then the following assertions are equivalent.
\begin{enumerate}
\item[(i)] $F$ is bisymmetric.
\item[(ii)] $F$ is associative and symmetric.
\item[(iii)] $F$ is ultrabisymmetric.
\end{enumerate}
\end{corollary}

\begin{proof}
We have (i) $\Rightarrow$ (ii) by Proposition~\ref{prop:20gt}. We have (ii) $\Rightarrow$ (iii) by Proposition~\ref{prop:21ft}. Finally we have (iii) $\Rightarrow$ (i) by Proposition~\ref{prop:ds5}.
\end{proof}

\begin{remark}
If $F\colon X^n\to X$ is bisymmetric and does not have a neutral element, then it need not be associative nor symmetric (e.g., $F(x,y,z)=x+2y+3z$ when $X=\R$). If $F\colon X^n\to X$ is ultrabisymmetric and does not have a neutral element, then it need not be associative (e.g., $F(x,y,z)=2x+2y+2z$ when $X=\R$).
\end{remark}

%---------------------------------------------------------------------------------------
\section{Operations on finite chains}

We now consider the special case when $X$ is a finite chain. Without loss of generality we will only consider the $k$-element chains $L_k=\{1,\ldots,k\}$, $k\geq 1$, endowed with the usual ordering relation $\leq$. It is known (see, e.g., \cite{BerPer06}) that there are exactly $2^{k-1}$ linear orderings $\preceq$ on $L_k$ that are single-peaked for $\leq$.

\begin{corollary}\label{cor:f456dfs2}
Let $F\colon L_k^n\to L_k$ be an operation. The following assertions are equivalent.
\begin{enumerate}
\item[(i)] $F$ is quasitrivial, symmetric, nondecreasing, and associative (associativity can be ignored when $n=2$).
\item[(ii)] $F$ is an idempotent $n$-ary uninorm.
\item[(iii)] There exists a linear ordering $\preceq$ on $L_k$ that is single-peaked for $\leq$ such that
\begin{equation}\label{eq:conjugMax5}
F(x_1,\ldots,x_n) ~=~ x_1\vee_{\preceq}\cdots\vee_{\preceq} x_n{\,},\qquad x_1,\ldots,x_n\in L_k.
\end{equation}
\end{enumerate}
If any of these assertions is satisfied, then $F$ has the neutral element $a_1$, where $a_1$ is the minimal element of $(L_n,\preceq)$. Also, there are exactly $2^{k-1}$ such operations.
\end{corollary}

\begin{proof}
(i) $\Rightarrow$ (iii). This implication follows from Theorem~\ref{thm:f456dfs}.

(iii) $\Rightarrow$ (ii). This is immediate since the minimal element of $(L_n,\preceq)$ is the neutral element of $F$.

(ii) $\Rightarrow$ (i). This implication follows from Corollary~\ref{cor:fs6}.
\end{proof}

\begin{remark}
By Corollary~\ref{cor:24f} (resp.\ Corollary~\ref{cor:24f1}) we can replace associativity with bisymmetry (resp.\ associativity and symmetry with ultrabisymmetry) in Corollary~\ref{cor:f456dfs2}.
\end{remark}

It is easy to see that any single-peaked linear ordering $a_1\prec\cdots\prec a_k$ on $L_k$ can be constructed as follows.
\begin{enumerate}
\item[1.] Choose $a_1\in L_k$.
\item[2.] For $i=2,\ldots,k$, choose for $a_i$ a closest element to the set $C_{i-1}$ in $L_k\setminus C_{i-1}$, where $C_i=\{a_1,\ldots,a_i\}$.
\end{enumerate}
From this observation we can now provide a graphical characterization of the idempotent $n$-ary uninorms $F\colon L_k^n\to L_k$ in terms of their contour plots. Recall that the contour plot of any operation $F\colon L_k^n\to L_k$ is the undirected graph $(L_k^n,E)$, where
$$
E ~=~ \{\{\bfx,\bfy\}\mid\bfx\neq\bfy~\text{and}~F(\bfx)=F(\bfy)\}.
$$

\begin{theorem}\label{thm:GC}
The following algorithm outputs the contour plot of an arbitrary idempotent $n$-ary uninorm $F\colon L_k^n\to L_k$.

\smallskip
\begin{enumerate}
  \item[Step 1.] Choose the neutral element $a_1\in L_k$ and set $C_1=\{a_1\}$. The point $(n\Cdot a_1)$ is necessarily isolated for $F$ with value $a_1$
  \item[Step 2.] For $i=2,\ldots,k$
  \begin{enumerate}
  \item[1.] Pick a closest element $a_i$ to $C_{i-1}$ in $L_k\setminus C_{i-1}$
  \item[2.] Set $C_i=\{a_i\}\cup C_{i-1}$
  \item[3.] Connect all the points in $C_i^n\setminus C_{i-1}^n$ with common value $a_i$
  \end{enumerate}
\end{enumerate}
\end{theorem}

\begin{proof}
The algorithm provides a single-peaked linear ordering $a_1\prec\cdots\prec a_k$ on $L_k$ together with the operation $F\colon L_k^n\to L_k$ defined as
$$
F(x_1,\ldots,x_n) ~=~
\begin{cases}
a_1, & \text{if $a_1\in\{x_1,\ldots,x_n\}$ and $a_2,\ldots,a_k\notin\{x_1,\ldots,x_n\}$},\\
a_2, & \text{if $a_2\in\{x_1,\ldots,x_n\}$ and $a_3,\ldots,a_k\notin\{x_1,\ldots,x_n\}$},\\
& \vdots\\
a_k, & \text{if $a_k\in\{x_1,\ldots,x_n\}$}.
\end{cases}
$$
This means that $F$ has precisely the form \eqref{eq:conjugMax5}.
\end{proof}

Figure~\ref{fig:fs2} shows the contour plot of an idempotent binary uninorm on $L_6$. To simplify the figure, we have omitted edges obtained by transitivity (i.e., connected points are joined by paths). The value shown on each path indicates the corresponding value.

\setlength{\unitlength}{3.5ex}
\begin{figure}[htbp]
\begin{center}
\begin{small}
\begin{picture}(8,8)
\put(0.5,0.5){\vector(1,0){7}}\put(0.5,0.5){\vector(0,1){7}}
\multiput(1.5,0.45)(1,0){6}{\line(0,1){0.1}}%
\multiput(0.45,1.5)(0,1){6}{\line(1,0){0.1}}%
\put(1.5,0){\makebox(0,0){$1$}}\put(2.5,0){\makebox(0,0){$2$}}\put(3.5,0){\makebox(0,0){$3$}}
\put(4.5,0){\makebox(0,0){$4$}}\put(5.5,0){\makebox(0,0){$5$}}\put(6.5,0){\makebox(0,0){$6$}}
\put(0,1.5){\makebox(0,0){$1$}}\put(0,2.5){\makebox(0,0){$2$}}\put(0,3.5){\makebox(0,0){$3$}}
\put(0,4.5){\makebox(0,0){$4$}}\put(0,5.5){\makebox(0,0){$5$}}\put(0,6.5){\makebox(0,0){$6$}}
\multiput(1.5,1.5)(0,1){6}{\multiput(0,0)(1,0){6}{\circle*{0.2}}}
\drawline[1](6.5,1.5)(1.5,1.5)(1.5,6.5)\drawline[1](2.5,6.5)(6.5,6.5)(6.5,2.5)\drawline[1](2.5,5.5)(2.5,2.5)(5.5,2.5)
\drawline[1](3.5,5.5)(3.5,3.5)(5.5,3.5)\drawline[1](4.5,5.5)(5.5,5.5)(5.5,4.5)
\put(1.85,1.85){\makebox(0,0){$1$}}\put(2.85,2.85){\makebox(0,0){$2$}}\put(3.85,3.85){\makebox(0,0){$3$}}
\put(4.85,4.85){\makebox(0,0){$4$}}\put(5.85,5.85){\makebox(0,0){$5$}}\put(6.85,6.85){\makebox(0,0){$6$}}
\end{picture}
\end{small}
\caption{An idempotent binary uninorm on $L_6$ (contour plot)}
\label{fig:fs2}
\end{center}
\end{figure}
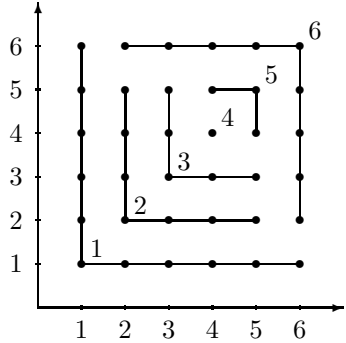

\begin{remark}
We remark that the binary versions of Corollary~\ref{cor:f456dfs2} and Theorem~\ref{thm:GC} were established in \cite{CouDevMar} by means of elementary proofs without using Proposition~\ref{prop:ack}.
\end{remark}

We end this section by the following alternative description of the class of idempotent $n$-ary uninorms.

\begin{theorem}\label{thm:Deb1}
An operation $F\colon L_k^n\to L_k$ with a neutral element $e$ is an idempotent uninorm if and only if there exists a nonincreasing map $g\colon\{1,\ldots,e\}\to\{e,\ldots,k\}$ (nonincreasing means that $g(x)\geq g(y)$ whenever $x\leq y$), with $g(e)=e$, such that
$$
F(x_1,\ldots,x_n) ~=~
\begin{cases}
\bigwedge_{i=1}^nx_i, & \text{if $\bigvee_{i=1}^nx_i\leq\overline{g}(\bigwedge_{i=1}^nx_i)$ and $\bigwedge_{i=1}^nx_i\leq\overline{g}(1)$},\\
\bigvee_{i=1}^nx_i, & \text{otherwise},
\end{cases}
$$
where $\overline{g}\colon L_k\to L_k$ is defined by
$$
\overline{g}(x) ~=~
\begin{cases}
g(x), & \text{if $x\leq e$},\\
\max\{z\in\{1,\ldots,e\}\mid g(z)\geq x\}, & \text{if $e\leq x\leq g(1)$},\\
1, & \text{if $x>g(1)$}.
\end{cases}
$$
\end{theorem}

\begin{proof}
This result was established when $n=2$ in \cite[Theorem~3]{DeBFodRuiTor09}. The general $n$-ary version then follows from Theorem~\ref{thm:f456dfs}.
\end{proof}

%---------------------------------------------------------------------------------------------- Acknowledgments
\section*{Acknowledgments}

The authors would like to thank Bruno Teheux for fruitful discussions and valuable remarks. This research is supported by the Internal Research Project R-AGR-0500 of the University of Luxembourg and the Luxembourg National Research Fund R-AGR-3080. The second author is also supported by the Hungarian National Foundation for Scientific Research, Grant No. K124749.

\end{document}